\documentclass[11pt,a4paper,english]{amsart}
\usepackage{graphicx}
\usepackage[left=25mm, right=25mm, top=25mm, bottom=25mm]{geometry}
\usepackage{color}
\usepackage{multicol}
\usepackage{mathrsfs}
\usepackage{indentfirst, amsfonts, amsmath, amsthm, amscd,amssymb}
\usepackage{epsfig}
\usepackage{hyperref}
\usepackage{amsmath}
\usepackage{amsthm}
\usepackage{amssymb}
\usepackage{amscd}
\usepackage{amsfonts}

\newcommand{\N}{\mathbb{N}} 
 
\newtheorem{theorem}{{Theorem}}[section]

\newtheorem{corol}[theorem]{{Corollary}}

\newtheorem{defin}[theorem]{{Definition}}

\title[A delay nonlocal quasilinear Chafee-Infante problem]{A delay nonlocal quasilinear Chafee-Infante problem: An approach via semigroup theory }
\author{Tom\'{a}s Caraballo, A. N. Carvalho and Yessica Julio }
\address[TC]{Depto. Ecuaciones Diferenciales y Anal. Num.\\
Facultad de Matem\'{a}ticas\\
Universidad de Sevilla\\
C/ Tarfia s/n\\
41012-Sevilla (Spain)}
\email[TC]{caraball@us.es}
\address[ANC and YJ]{
Instituto de Ci\^{e}ncias Ma\-te\-m\'{a}\-ti\-cas e de Computa\c{c}\~{a}o 
Universidade de S\~{a}o Paulo, Campus de S\~{a}o Carlos, Caixa Postal 668, S\~{a}o Carlos SP, Brazil.}
\email[ANC]{andcarva@icmc.usp.br}
\email[YJ]{yessica.julio@usp.br}
\thanks{[TC] Partially supported by the Spanish Ministerio de Ciencia e Innovaci{\'o}n (MCI), Agencia Estatal de Investigaci{\'o}n (AEI) and Fondo Europeo de Desarrollo Regional (FEDER) under the project PID2021-122991NB-C21.}
\thanks{[ANC] Partially supported by FAPESP Grant \# 20/14075-6 and by CNPq Grant \# 308902/2023-8, Brazil}
\thanks{[YJ] Partially supported  by CAPES Grant \# 88887.695331/2022-00 and  by the Colombian Ministerio de Ciencia, Tecnolog\'{\i}a e Innovaci\'{o}n (Minciencias).}



\begin{document}
	
	\maketitle
	
	\thispagestyle{empty} 
	
\begin{abstract}
In this work we study a dissipative one dimensional scalar parabolic problem with non-local nonlinear diffusion with delay. 
We consider the general situation in which the functions involved are only continuous and solutions may not be unique. We establish conditions for global existence and prove the existence of 
global attractors. All results are presented only in the autonomous since the non-autonomous case follows in the same way, including the  existence of pullback attractors. A particularly interesting feature
is that there is a semilinear problem (nonlocal in space and in time) from which one can obtain all solutions of the associated quasilinear problem and that for this semilinear problem the delay depends on the initial function making its study more involved.
\end{abstract}
	
\section{Introduction}
	
Reaction-diffusion equations with non-local terms have attracted great attention during the last twenty years. A few representative references are \cite{chipot1999asymptotic,michel2001asymptotic,chang2003nonlinear,chipot2003remarks,CLLM2020,caballero2021existence,CM2021,ACM2024}).  \\
In order to explain the problems we wish to consider, let us start with an example of the type of models we have in mind, that is, consider the following quasilinear initial value problem with delay:
\begin{equation}\label{problema}
\left\{ 
\begin{split}
&\dfrac{\partial w}{\partial \tau} -a(l(w)) \dfrac{\partial w^2}{\partial x^2}=\lambda f(w) + \gamma w(\tau -\rho) + h(\tau), \  \tau > 0, \ x \in \Omega, \\
&w(\tau,0)=w(\tau,1)=0,\\
&w(\tau)=\phi(\tau), \, \tau \in [- \rho,0],
\end{split}
\right.
\end{equation}
where $\Omega=(0,1)$, $\gamma, \lambda , \rho> 0$, $l$ a continuous operator from $H^{1}_{0}(\Omega)$ into $\mathbb{R^+}$, $a \in C(\mathbb{R})$  with $a(\mathbb{R})\subset [m,M]\subset (0,\infty)$, $f\in C(\mathbb{R},\mathbb{R})$, $h\in C(\mathbb{R},L^2(0,1))$ is bounded and $\phi \in C([-\rho,0],H^{1}_{0}(\Omega))$. \\
The local problem, i.e. $a\equiv 1$, without delay, has first shown to have very interesting properties in $1974$  in the seminal works of N. Chafee and E. Infante (see \cite{chafee-infante,chafee1974bifurcation}). Through the work of many authors, this has become the best understood infinite dimensional dynamical system  (see, for example, \cite{henry1981geometric,FuscoRocha1991} for the autonomous case and \cite{BrocheCarvalhoValero2019} for the non-autonomous case).
\\
For the non-local problem ($a$ non-constant), without delay, many interesting new features have been discovered relative to what was known for the local case making this problem a very interesting one from the point of view of dynamics
(see for example \cite{ACM2024,CM2021,CLLM2020,caballero2021existence}). Of course,  non-local problems are quite challenging from the analytical point of view making any new discovery even more interesting.\\
In \cite{CCJ2024} we have dealt with a prototype of non-autonomous scalar one dimensional parabolic problem with the non-local nonlinear diffusion being only continuous and through the semigroup theory.  The introduction of the time variable makes the problem quite challenging and interesting already. \\
In this paper we go one step further considering models similar to those treated in \cite{CCJ2024} with delay. As we will see next, this brings up a new and interesting feature to the problem, that is, the initial value problem has to be considered with a delay depending on the initial function.\\
To obtain the local existence and regularity  of mild solutions (continuous in time functions taking value in a suitable phase space and satisfying the variation of constants formula) for \eqref{problema} we will use the results of \cite{CCJ2024} but, to that end, we will need to deal with the very interesting new feature of problems with delay depending on the initial function  $\phi$. To obtain regularity, some additional assumption is needed on $f$, $a\circ l$ (as in \cite{CCJ2024}) but also on the initial function $\phi$.  \\
To ensure that solutions are globally defined and to be able to apply the method of steps we impose the structural condition

\bigskip

\noindent (\textbf{S}) Assume that there exist $C_0,C_1 \in \mathbb{R}$ such that 
$$
 uf(u) \leqslant -\nu C_0u^2 + |u| C_1
$$
for all $u\in \mathbb{R}$ and for both $\nu=\frac{m}{\lambda}$ and $\nu=\frac{M}{\lambda}$.

\bigskip

Finally, to obtain the existence of a global attractor we assume, the dissipativity condition 

\bigskip

\noindent (\textbf{D}) Assume that (\textbf{S}) holds for some $C_0$ such that the first eigenvalue $\omega$ of $A+\nu C_0I$ is positive and satisfies  $e^{-\omega \rho/m}+\frac{\gamma}{\omega m}<1$.

\bigskip

The nonlinear nonlocal diffusion $a(l(\cdot )):H^1_0(0,1)\to [m,M]$ makes the above problem a nonlocal (in space) quasilinear problem. 

\bigskip

Our aim will be to establish a general local existence and regularity result for solutions to \eqref{problema}, prove that if condition $\mathbf{(S)}$ is satisfied, solutions are globally defined and, if  condition $\mathbf{(D)}$ is satisfied, the multivalued semiflow associated to \eqref{problema} has a global attractor. We also use comparison results to obtain uniform bounds for the solutions in the global attractor. We could work with the nonlinearity as in \eqref{problema} (being time dependent) but all results would have identical proofs and therefore we have decided to consider only the case $h\equiv 0$. It will be clear from the proofs that adding a bounded continuous function $h:\mathbb{R} \to L^2(0,1)$,or even more general non-autonomous nonlinearities, will not change the proofs so we choose to omit it for the sake of simplicity in the notation.

\bigskip

Before we proceed, let us work a little more with the model \eqref{problema} in order to understand the interesting new feature it brings. Given a solution $w:[-\rho,\infty)\to H^1_0(0,1)$ of the problem \eqref{problema}, making $t\!:=\!\alpha^{-1}\!(\tau)\!={\displaystyle\int_0^\tau }a(l(w(r)))^{-1}dr$, $\tau\in[-\rho,\infty)$,  we have that the function $u$ defined by $u(t)=w(\tau)$ will be a solution of the  problem
\begin{equation}\label{problemau}
\left\{ \begin{array}{lcc}
u_t = u_{xx} + \dfrac{\lambda f(u) + \gamma u(\alpha(t)-\rho)}{a(l(u))}, \,\,\,\,\,\,\,\,\,\, \thinspace t >0, \,\, x \in \Omega,  \\
\\ u(t,0)=u(t,1)=0,\\
\\ u(t)=\phi(\alpha_0(t)), \ t \in [\alpha^{-1}(-\rho),0],
\end{array}
\right.
\end{equation}
where $\alpha^{-1}(-\rho)\!=\!{\displaystyle\int_0^{-\rho}} \!\!a(l(\phi(r)))^{-1}\! dr\!<\!0$ and, since $a(\mathbb{R})\!\subset \![m,M]\!\subset\! (0,\infty)$, $\frac{\rho}{M}\leqslant -\alpha^{-1}(-\rho)\leqslant \frac{\rho}{m}$.

We note that, if $\tau\geqslant 0$  
$$
\alpha(t)={\displaystyle\int_{0}^{t}}a(l(u(r))dr
$$ 
and  $\alpha_0:[\alpha^{-1}(-\rho),0]\to  [-\rho,0]$ is the only solution of the integral equation
$$
\alpha(s)={\displaystyle\int_0^{s}}a(l(\phi\circ\alpha(r)))dr, \ s\in (\alpha^{-1}(-\rho),0].
$$

On the other hand, given $\phi \in C(-\rho,0],H^{1}_{0}(\Omega))$ we define $\alpha_0$ and $\alpha^{-1}(\rho)$ as above. If $u:[\alpha^{-1}(-\rho),\infty)\to H^1_0(0,1)$ is a solution of the semilinear delay differential problem \eqref{problemau},
making 
$$
\tau =\alpha(t)= {\displaystyle\int_{0}^t}a(l(u(r)))dr,
$$ 
the function $w(\tau):=u(t)$ will be a solution of problem \eqref{problema}. 

\bigskip

Problem \eqref{problemau} is a non-local (in time and in space) non-autonomous semilinear problem.

\bigskip

Even though, for a fixed initial condition $\phi \in C([-\rho,0],H^{1}_{0}(\Omega))$, the solution of the problem \eqref{problemau} may not be unique, the delay is always the same, determined by the initial function $\phi$. In fact, it is a striking feature of this model that $\alpha_0$ and $\alpha^{-1}(\rho)$ are uniquely determined by $\phi$ only and that, to solve the \eqref{problema} with initial function $\phi$, corresponds to solve the nonlocal non-autonomous semilinear delay differential problem \eqref{problemau} with delay determined by $\phi$.

\bigskip

With this, we will prove the existence of solutions to problem \eqref{problemau} and
those will give us  solutions for \eqref{problema}. We will use the method of steps which consists of solving the problem iteratively, in intervals of time of length $\displaystyle\alpha^{-1}(\rho)=\int_{-\rho}^0 \!a(l(\phi(r)))^{-1}\! dr$.  In each step we will apply the results obtained in \cite{CCJ2024} for the semilinear non-local (in time and in space) non-autonomous problem \eqref{problemau}. As in \cite{CCJ2024}, we will proceed  as abstract as possible in order that the theory can be applied to other similar models with little effort.

\bigskip

Let us now recall the results of \cite{CCJ2024}. For a Banach space  $X$, let $C(X)$ ($L(X)$) denote the space of continuous (linear continuous) transformations from $X$ into itself.  For $u\in C([0,T],X^\alpha)$, let $u^t(\cdot) = u(\cdot)\big|_{[0,t]}$, $t\in [0,T]$. Consider the Cauchy problem
\begin{equation}\label{general}
\left\{ \begin{array}{lcc}
\dfrac{ du(t)}{dt}+  Au(t)= g(t,u^t(\cdot))  \,\,\,\,\,\,\,\,\,\, \thinspace t>0, \,\, \\
\\ u(0)=w_0 \in X^\alpha,
\end{array}
\right.
\end{equation}
where $-A$ has compact resolvent and is the infinitesimal generator of an exponentially decaying analytic semigroup $\{e^{-At}:t\geqslant 0\}$, $X^\alpha, \: \alpha \geqslant 0$, is the fractional power spaces associated to $A$ \cite{henry1981geometric,komatsu1966fractional}, and $g:[0,T] \times C([0,T], X^\alpha) \rightarrow X$ is a continuous function.\\
With this preliminaries,  we introduce the definitions of strong and mild solution for \eqref{general}.

\begin{defin}\cite{pazy}
A function $u:[0,T) \rightarrow X$ is a strong solution of \eqref{general} on $[0,T),$ if $u$ is continuous on $[0,T)$ and continuously differentiable on $(0,T), \: u(t) \in D(A)$ for $0<t<T,~u(0)=w_0$  and \eqref{general} is satisfied on $[0,T).$

\end{defin}

\begin{defin}\cite{pazy}
Let $A$ be the infinitesimal generator of an analytic semigroup $S(t)$, $w_0 \in X^\alpha $ and $g :[0,T) \times  C([0,T], X^\alpha) \rightarrow X$.
A function $ u \in C([0,T]; X^\alpha) $ such that
\begin{equation}\label{solufraca}
u(t)=S(t)w_0 + \int_0^t S(t-s)g(s,u^s(\cdot)) ds, \qquad 0 \leqslant t \leqslant T,
\end{equation}
is called a mild solution of the initial value problem \eqref{general} on $[0,T].$
\end{defin}

We summarize the results of \cite{CCJ2024} next.

\begin{theorem}[\cite{CCJ2024}]\label{mildsolution}
Let $X$ be a Banach space, $A$ a sectorial operator such that $(\lambda + A)^{-1}$ is compact for all $\lambda \in \rho (-A)$, and let $\{S(t): \,t\geqslant 0\}$ be the analytic semigroup generated by $-A$. If $g:[0,T]\times C([0,T],X^\alpha) \rightarrow X, \, 0\leqslant \alpha \leqslant 1$ is a continuous map,
then for each  $w_0\in X^\alpha$ there exists a $T_1=T_1(w_0)\in (0,T]$ such that the initial value problem \eqref{general} has a mild solution $u \in C([0,T_1];X^\alpha)$. Furthermore, $T_1$  may be chosen uniformly for $w_0$ in bounded subsets of $X^\alpha$.

In addition, if $g: [0,T] \times   C([0,T], X^{\alpha}) \rightarrow X, \, 0\leqslant \alpha \leqslant 1$, is such that, given $u \in C([0,T], X^{\alpha}) $ 

\begin{equation}\label{hipsobref}
\|g(t,u^t(\cdot))-g(s,u^s(\cdot))\|_X \leqslant w(\|u(t)-u(s)\|_{\alpha})+w(|t-s|^\beta),\,\,0 <\beta<1-\alpha,
\end{equation}
where $w:[0,\infty) \rightarrow [0,\infty)$ is an increasing continuous function such that $w(0)=0$ and 
\begin{equation}\label{hipsobrew}
\displaystyle \int_{0}^{t}u^{-1}w(u^\beta)du < \infty .
\end{equation}

A continuous function $u:[0,T_1] \rightarrow X^{\alpha}$  satisfying
$$
u(t)=S(t)w_0 + \displaystyle \int_{0}^{t} S(t-s) g(s,u^s(\cdot))ds , \;\;\ \;0 \leqslant t \leqslant T_1,
$$
is a strong solution of \eqref{general}.
\end{theorem}

Next we will show how to apply the results in \cite{CCJ2024} to establish existence and regularity of solutions for \eqref{problemau}. Of course, since $f:\mathbb{R}\to \mathbb{R}$, $a:\mathbb{R}^+\to [m,M]\subset (0,\infty)$, $l:H^1_0(0,1)\to \mathbb{R}$ and $\phi \in C([-\rho,0],H^1_0(0,1))$ we have that, for $T\in [0,\alpha^{-1}(\rho))$,
$$
g(t,u^t(\cdot))(x)= \dfrac{\lambda f(u(t)(x)) + \gamma \phi(\alpha(t)-\rho)}{a(l(u^t(t)))}
$$ 
is a continuous map from $[0,T]\times C([0,T],H^1_0(0,1)$ into $X=L^2(0,1)$. The only point that has to be analyzed more carefully is the continuity of $[0,\alpha^{-1}(\rho)]\ni t \mapsto \phi(\alpha(t)-\rho)\in L^2(0,1)$, but that follows from the continuity of $\phi$ and from the fact that
$$
|\alpha(t)-\alpha(t')|=|\int_{t'}^t a(l(u(r)))dr|\stackrel{t\to t'}{\longrightarrow} 0.
$$

It follows that the following theorem holds.
 
\begin{theorem}\label{mild_problemau}
Let $X=L^2(0,1)$, $D(A)=H^2(0,1)\cap H^1_0(0,1)$ and $A:D(A)\subset X\to X$ be the operator defined by $Au=u_{xx}$. Then $-A$ is positive and self-adjoint (hence sectorial), with fractional power spaces $X^\alpha:=D({-A}^\alpha)$ with the graph norm, $\alpha\geqslant 0$, $X^\frac12=H^1_0(0,1)$,  $(\lambda + A)^{-1}$ has compact resolvent, $\lambda \in \rho (-A)$ and the semigroup $\{S(t): \,t\geqslant 0\}$ generated by $A$ is a compact and exponentially decaying analytic semigroup. If $f: \mathbb{R} \to \mathbb{R}$ is a continuous function satisfying $\mathbf{(S)}$, given $\phi\in C([-\rho,0],H^1_0(0,1))$ the initial value problem \eqref{problemau} has a mild solution $u \in C([0,\infty);H^1_0(0,1))$, furthermore  $u \in C([0,\infty);X^\alpha)$, for all $\alpha\in (0,1)$. 
\end{theorem}

For regularity, we assume also that $l:H^1_0(0,1)\to \mathbb{R}$ is continuous  and that $f:\mathbb{R}\to \mathbb{R}$, $a:\mathbb{R}^+\to [m,M]$, $l:H^1_0(0,1)\to \mathbb{R}$ and $\phi:(-\rho,0]\to H^1_0(0,1)$ satisfy

\begin{equation}\label{modulusofcont}
\begin{split}
&| f(s)-f(s')| \leqslant w(|s-s'|) \\
&| a(l(u))-a(l(v))| \leqslant w(\|u-v\|_{H^1_0(0,1)}) ,
\\
&|h(s)-h(s')| \leqslant w(|s-s'|^\beta),\\
&\|\phi(t)-\phi(t')\|_{X} \leqslant w(|t-t'|^\beta),
\end{split}
\end{equation}
where $w:[0,\infty) \rightarrow [0,\infty)$ is a continuous increasing function, $w(0)=0$ and
$$
 \displaystyle \int_{0}^{1}u^{-1}w(u^{\beta})du < \infty,
\hbox{ for some }\beta\in (0,1-\alpha).
$$ 
We only need to pay attention to the function $[0,\alpha^{-1}(\rho)]\ni t \mapsto \phi(\alpha(t)-\rho)\in L^2(0,1)$ and check that, under the above conditions, this function satisfies
the conditions of Theorem \ref{mildsolution} for the existence of a mild solution.

That follows from
\begin{equation*}
\begin{split}
\|\phi(\alpha(t)-\rho)-\phi(\alpha(t')-\rho)\|_{X}&\leqslant
w(|\alpha(t)-\alpha(t')|)\\
&=w(|\int_{t'}^t a(l(u(r)))dr|)\leqslant w(M|t-t'|).
\end{split}
\end{equation*}

\begin{theorem}\label{T_regularity}
Under the assumptions of Theorem 	\ref{mild_problemau} and that \eqref{modulusofcont} is satisfied, the initial value problem \eqref{problema} has a strong solution $w \in C([0,\infty);H^2(0,1)\cap H^1_0(0,1))$. 
\end{theorem}

\section{Existence of Solutions}
\begin{proof}[Proof of Theorem \ref{T_regularity}]
We are going to use the method of steps. Let us review the reasoning and the needed details to apply Theorem \ref{mildsolution}.
We know that $w(\tau)=\phi(\tau)$ for $\tau \in (-\rho,0]$. Then, for $\tau \in (0,\rho]$ , $\tau-\rho\in (-\rho,0]$ and then we have $w(\tau-\rho)=\phi(\tau  -\rho)$, so that equation \eqref{problema} becomes the following non-autonomous quasilinear non-local scalar one-dimensional parabolic partial differential equation		
\begin{equation}\label{Eq01}
\left\{ 
\begin{split}
&\dfrac{\partial w}{\partial \tau} =a(l(w)) \dfrac{\partial w^2}{\partial x^2}+\lambda f(w) + \gamma \phi(\tau - \rho),\  \tau>0\ ,x \in \Omega,  \\ 
&w(\tau,0)=w(\tau,1)=0,\\
&w(0)=\phi(0).
\end{split}
\right.
\end{equation}
One can perform a change in the time scale in order to obtain the semilinear problem 
\begin{equation}\label{problemapaso1}
\left\{ 
\begin{split}
&u_t = u_{xx} + \dfrac{\lambda f(u) + \gamma \phi(\alpha(t)- \rho)}{a(l(u))}, \  t >0, \,\, x \in \Omega,  \\
& u(t,0)=u(t,1)=0,\\
& u(0)=\phi(0),
\end{split}
\right.
\end{equation}
where $\tau = {\displaystyle\int_{0}^t} a(l((u(r)))dr=:\alpha(t)$.

It is clear, from the discussion in the Introduction, that this initial value problem has at least a mild solution that we will call $g_1$ (see \cite{CCJ2024}). Under Assumption $\mathbf{(S)}$, since $[0,\alpha^{-1}(\rho)]\ni t\mapsto \phi(\alpha(t)- \rho)\in H^1_0(0,1)$ is continuous, this solution exists in the interval $[0,\alpha^{-1}(\rho)]$. Thus, we just need to check regularity. 

To that end, set $f_1:[0, \alpha^{-1}(\rho)] \times C([0,\alpha^{-1}(\rho)],X^\alpha) \rightarrow X$ defined by 
$$
f_1(t,u^t(\cdot))=\dfrac{\lambda f(u(t))+\gamma \phi (\alpha(t)-\rho)}{a(l(u(t)))}, \,\, \alpha(t)= \int_{0}^t a(l(u(\theta)))d\theta.
$$ 
Then,  for $g_1(t),~g_1(s) \in V,$ for $V$ a neighborhood of $\phi(0),$ we have that
\begin{equation}
    \begin{split}  \|f_1(t,&g_1^t(\cdot))  \!-\! f_1(s,g_1^t(\cdot))\| \\
    & \leqslant \left \|\dfrac{\lambda f(g_1(t))\!+\!\gamma \phi (\alpha(t)\!-\!\rho)}{a(l(g_1(t)))}\! -\! \dfrac{\lambda f(g_1(s))\!+\!\gamma \phi (\alpha(s)\!-\!\rho)}{a(l(g_1(s)))} \pm\! \dfrac{\lambda f(g_1(s))\!+\!\gamma \phi (\alpha(s)\!-\!\rho)}{a(l(g_1(t)))}\right\| \\
       &  \leqslant m^{-1}\left\| \lambda(f(g_1(t))-f(g_1(s)))\right \| +m^{-1}\left\| 
      \gamma(\phi(\alpha(t)-\rho)-\phi(\alpha(s)-\rho)) \right \| \\
       & +  \| (\lambda f(g_1(s)) + \gamma \phi(\alpha(s)-\rho))\| \left |\dfrac{a(l(g_1(s)))-a(l(g_1(t)))}{a(l(g_1(t)))a(l(g_1(s)))}\right |.
    \end{split}
\end{equation}

Since for $t>s$ we have 
\begin{equation}
\begin{split} |\alpha(t)-\alpha(s)| & 
= \left| \displaystyle \int_s^t a(l(g_1(\theta)))d\theta  \right | \\
& \leqslant M|t-s|. 
\end{split}
\end{equation}
It follows that 
\begin{equation}
\|f_1(t,g_1^t(\cdot))  - f_1(s,g_1^s(\cdot))\|  \leqslant K_1 w(\|g_1(t) - g_1(s) \|_\alpha) +  K_2 w(c |t-s|^\beta),
\end{equation}
for some positive numbers $K_1,K_2$ and $c$, and since
$$
\displaystyle \int_{0}^{t}u^{-1}w(u^{\beta})du < \infty,\,\,for\,\, 0< \beta < 1- \alpha,
$$
we can apply Theorem \ref{mildsolution} and conclude
that $g_1$ is a strong solution to equation \eqref{problemapaso1} in the interval $[0,\alpha^{-1}(\rho)]$. Also $w(\tau)=g_1(t)$ is a solution of \eqref{Eq01} in the interval $[0,\rho]$.

Now, for $\tau \in [\rho, 2\rho], \, \tau -\rho \in [0,\rho]$ and we have that   $w(\tau)=u(\alpha^{-1}(\tau))= g_1(\tau)$, so  equation \eqref{problema} becomes
		\begin{equation}\label{problema3}
			\left\{ \begin{array}{lcc}
				\dfrac{\partial w}{\partial \tau} =a(l(w)) \dfrac{\partial w^2}{\partial x^2}+\lambda f(w) + \gamma g_1(\tau-\rho), \,\,\,\,\,\,\,\,\,\, \thinspace \tau > \rho,  \\
    \\ w(\tau,0)=w(\tau,1)=0,\\
				\\ w(\rho)=g_1( \rho).
			\end{array}
			\right.
		\end{equation} 
Again, by making a change in the time scale we have  the initial value problem
\begin{equation}\label{problemapaso2}
			\left\{ \begin{array}{lcc}
				\dfrac{\partial u}{\partial t} =\dfrac{\partial u^2}{\partial x^2}+\dfrac{\lambda f(u) + \gamma g_1(\alpha(t)-\rho)}{a(l(u))}, \ t \in (\alpha^{-1}(\rho),\alpha^{-1}(2\rho)],  \\
    \\ u(t,0)=u(t,1)=0,\\
				\\ u(\alpha^{-1}(\rho))=g_1( \alpha^{-1}(\rho)).
			\end{array}
			\right.
		\end{equation} 
Since $g_1 \in C((0,\alpha^{-1}(\rho)],X^{\alpha})$, we have that the application $$
f_1(t,u(t),u(\cdot))=\dfrac{\lambda f(u(t))+\gamma g_1((\alpha(t)-\rho))}{a(l(u(t))}
$$ 
is continuous as long as $\alpha(t)\in [\rho,2\rho]$, which means that $t \in [\alpha^{-1}(\rho), \alpha^{-1}(2\rho)]$. Thus, again we just need to check the regularity of the solution.
 Hence, we need to see that
$$
\|f_1(t,u^t(\cdot))-f_1(s,u^s(\cdot)) \| \leqslant w(\|u(t)-u(s)\|_{\alpha}) + w(|t-s|^{\beta}),
$$ 
but that follows exactly as before. To proceed, we make $g_2(\tau)=u(\alpha^{-1}(\tau))$ for $\tau \in [\rho,2\rho]$.

Continuing in this way, if $g_{n-1}(\tau)=u(\alpha^{-1}(\tau))$ for $\tau \in [(n-2)\rho,(n-1)\rho]$ we have, for $n \geqslant 3$,
\begin{equation}\label{problemanesimo}
\left\{ \begin{split}
&\dfrac{\partial w}{\partial \tau} -a(l(w)) \dfrac{\partial w^2}{\partial x^2}=\lambda f(w) + \gamma g_{n-1}(\tau-\rho),\  \tau\in ((n-1)\rho, n\rho],  \\
&w(\tau,0)=w(\tau,1)=0,\\
& w((n-1)\rho)=g_{n-1}( (n-1)\rho).
\end{split}
\right.
\end{equation} 
has a mild solution $g_n:[(n-1)\rho, n\rho]\to H^1_0(0,1)$.
Now if $w\!:\![-\rho,\infty) \to H^1_0(0,1)$ is given by
\begin{equation}\label{solution}
w(\tau)=\left\{ \begin{split}
&\phi(\tau) , \,\, \tau\in [-\rho ,0] ;\\
& g_n(\tau), \,\, \tau \in ( (n-1)\rho,n \rho],~n \geqslant 1,
\end{split}
\right.
\end{equation}
it is a strong solution of \eqref{problema}.

Using the variation of constants formula we have 
$$
g_1(\tau)=u(t)= S(t)\phi(0) + \displaystyle \int_{0}^{t}S(t-s)\frac{\lambda f(g_1(\alpha(s)))+\gamma\phi(\alpha(s)-\rho)}{a(l(g_1(\alpha(s))))}ds, \ 0\leqslant t\leqslant \alpha^{-1}(\rho)
$$ 
$$
g_2(\tau)=u(t)=S(t -  \alpha^{-1}(\rho))g_1( \rho) + \displaystyle \int_{\alpha^{-1}(\rho)}^{t}S(t-s)\frac{(\lambda f(g_2(\alpha(s)))+\gamma g_{1}(\alpha(s)-\rho))}{a(l(g_2(\alpha(s))))}ds, 
$$ 
for $t\in [\alpha^{-1}(\rho), \alpha^{-1}(2\rho)]$.
We know that 
$$
g_1( \rho)= S( \alpha^{-1}(\rho))\phi(0) + \displaystyle \int_{0}^{ \alpha^{-1}(\rho)}S(  \alpha^{-1}(\rho) -s)\frac{\lambda f(g_1(\alpha(s)))+\gamma\phi(\alpha(s)-\rho)}{a(l(g_1(\alpha(s))))}ds,
$$
then, we obtain that 
$$
S(t-\alpha^{-1}(\rho)) g_1(\rho)=S(t)\phi(0) + \displaystyle \int_{0}^{\alpha^{-1}(\rho)}S(t -s)\frac{\lambda f(g_1(\alpha(s)))+\gamma\phi(\alpha(s)-\rho)}{a(l(g_1(\alpha(s))))}ds
$$
and
\begin{equation*}
\begin{split}
 g_2(\tau)=u(t)= S(t)\phi(0) & + \displaystyle \int_{0}^{\alpha^{-1}(\rho)}S(t -s)\frac{\lambda f(g_1(\alpha(s)))+\gamma\phi(\alpha(s)-\rho)}{a(l(g_1(\alpha(s))))}ds \\& 
 + \displaystyle \int_{\alpha^{-1}(\rho)}^{t}S(t -s) \frac{(\lambda f(g_2(\alpha(s)))+\gamma g_1(\alpha(s)-\rho))}{a(l(g_2(\alpha(s))))}ds. 
\end{split}
\end{equation*}
Similarly, for $n>1$ we obtain that

\begin{equation*}
\begin{split}
    g_n(\tau)& =u(t) =  S(t)\phi(0) + \displaystyle \sum_{m=1}^{n-1}\displaystyle \int_{  \alpha^{-1}((m-1)\rho)}^{ \alpha^{-1}( m\rho)}S(t -s)\frac{(\lambda f(g_m(\alpha(s)))+\gamma g_{m-1} (\alpha(s) -\rho))}{a(l(g_m(\alpha(s))))}ds \\ & 
 + \displaystyle \int_{ \alpha^{-1}((n-1)\rho)}^{t}S(t -s) \frac{(\lambda f(g_n(\alpha(s)))+\gamma g_{n-1}(\alpha(s)-\rho))}{a(l(g_n(\alpha(s))))}ds.
\end{split}
\end{equation*}
\end{proof}

 \section{Comparison results and global existence} 
Throughout this section, $\phi$ is the initial function  of problem \eqref{problema}, $K>0$ is such that $-K\leqslant \phi(t) \leqslant K$ and $``\leqslant"$ is a partial ordering in $H^{1}_{0}(\Omega)$, that is:
 $$
u \leqslant v ~ \text{in} ~ H^{1}_{0}(\Omega) \Leftrightarrow u(x) \leqslant v(x) ~ \text{a.e. ~ for} ~ x ~ \text{in} ~ \Omega.
$$
We consider the initial value problems :

\vspace{-20pt}

\begin{multicols}{2}
\begin{equation}\label{f-}
\begin{split}
	&\dfrac{\partial u}{\partial t}\!=\! Au \!+\!\dfrac{\lambda f(u) \!-\! K}{M},\, t\!>\!0, x\! \in\! \Omega, \\
              &u(0)=-K,\\
	& u|_{\partial\Omega}=0,
\end{split}
\end{equation}

\begin{equation}\label{f+}
\begin{split}
	&\dfrac{\partial u}{\partial t}\!=\! Au \!+\! \dfrac{\lambda f(u) \!+\! K}{m},\, t\!>\!0, x\! \in\! \Omega,  \\
              &u(0)=K,\\
              & u|_{\partial\Omega}=0,
\end{split}
\end{equation}
\end{multicols}
 and we write 
 \begin{equation*}
    \begin{array}{rl}
&  f^{-}(u)  =\dfrac{\lambda f(u) - \gamma K}{M},\quad  uf^{-}(u)\leqslant \nu C_0u^2 + \left(\frac{\gamma K}{M}+ C_1\right)|u|,\\\\
&  f^{+}(u)= \dfrac{\lambda f(u) +\gamma K }{m} ,\quad  uf^{+}(u)\leqslant \nu C_0u^2 + \left(\frac{\gamma K}{m}+ C_1\right)|u|,\\\\
& g(t,u^t(\cdot))=\dfrac{\lambda f(u(t)) +\gamma u( \int_{0}^t a(l(u(\theta)))-\rho )d\theta }{a(l(u(t)))},
         
\end{array}
\end{equation*}
\noindent where $A:D(A) \rightarrow L^2(\Omega)$ is the linear operator defined in the following way, $D(A)= H^2(\Omega) \cap  H^{1}_{0}(\Omega) $  and $Au=u_{xx}, ~ u \in D(A) $; $f:\mathbb{R}\to \mathbb{R}$ and $g:[0, T] \times  C([0,T], X^\alpha) \rightarrow X$ are continuous functions. For $u(\cdot) \in C([0, T], X^\alpha), u^{t}(\cdot) = u(\cdot)|_{[0,t]}, u(t)=u^{t}(t)$. We shall prove that, under some structural condition on $f$ and assuming that for each $r>0$ there is a $\kappa=\kappa(r)$ such that $u\mapsto \kappa u + f(u) $ is an increasing function in $[-r,r]$, then for each $n\in \mathbb{N}$, there is a constant $K_n$ such that the solutions $u(t,\phi)$ of \eqref{problemau} are globally defined, and there are $u(t,K_n),~u(t,-K_n),$ solutions of \eqref{f+} and \eqref{f-} with $K$ replaced by $K_n$, such that $u(t,-K_n) \leqslant u(t,\phi)\leqslant u(t,K_n)$, $t\in [\alpha^{-1}((n-1)\rho),\alpha^{-1}(n\rho)]$. 

\bigskip

That is, thanks to the fact that the solutions of \eqref{f-} and \eqref{f+} are globally defined, we can guarantee that the solutions of \eqref{problemau} are defined in the interval $[0,\alpha^{-1}(\rho)]$ and there is a positive constant $K_1$ such that $-K_1\leqslant u(t,\phi) \leqslant K_1$, $0\leqslant t \leqslant \alpha^{-1}( \rho)$. We can therefore  iterate this procedure to obtain that the solutions of \eqref{problemau} are globally defined.

\bigskip

To prove results described above, we need to impose the structural condition (\textbf{S}) stated in the introduction to the non-linear forcing term.
This condition ensures that the solutions of \eqref{f+} and \eqref{f-} are global and consequently, using the procedure described above we obtain that the solutions of \eqref{problema} are global.

\bigskip

Observe that  condition (\textbf{D}) imposes a restriction on the constant $C_0$ that appears in condition $\mathbf{(S)}$. It was used in \cite{CCJ2024} to ensure the existence of pullback attractor. Here some special care needs to be taken due to the procedure described above with changing constants $K$ at intervals of length $\rho$. In fact, the following result holds

\begin{theorem}\label{corolariocomparacion}
 Assume that (\textbf{S}) holds for a continuous function $f:\mathbb{R}\to \mathbb{R}$  such that, for every $r>0$, there exists a constant $\kappa= \kappa(r) > 0 $ such that $s\mapsto \kappa s+  f(s)$ is increasing in $[-r,r]$. Let $\phi:[-\rho,0]\to H^1_0(0,1)$ and $K>0$ such that $-K\leqslant \phi(\tau)(x)=\phi(\alpha(t))(x)\leqslant K$ for all $\tau\in [-\rho,0]$ and $x\in [0,1]$. If $u(t,\phi)$ is a solution of 
 \begin{equation}\label{problemaaux}
\left\{ 
\begin{split}
&u_t = u_{xx} +  g(t,u^t(\cdot)), \  t >0, \,\, x \in \Omega,  \\
& u(t,0)=u(t,1)=0,\\
& u(r)=\phi(\alpha(r)), r\in [\alpha^{-1}(-\rho),0],
\end{split}
\right.
\end{equation}
there are 
$u(t,-K)$, $u(t,K)$ solutions of \eqref{f-} and \eqref{f+} such that $u(t,-K) \leqslant u(t,\phi)\leqslant u(t,K)$, for $t\in [0,\alpha^{-1}(\rho)]$.  As an immediate consequence of this, $u(\cdot,\phi)$ is defined for all $t\geqslant 0$.
\end{theorem}
\begin{proof}
We will use the iterative step method. For that, first we consider  $t \in (0, \alpha^{-1}(\rho)],$ then, $\alpha(t) - \rho \in (- \rho,0],$  and $w(\tau- \rho)= \phi(\alpha(t)- \rho)$, so that the equation \eqref{problema} becomes the initial value problem
\begin{equation}\label{paso1.1}
\left\{ 
\begin{split}
&u_t = u_{xx} + \dfrac{\lambda f(u) + \gamma \phi(\alpha(t)-\rho)}{a(l(u))}, \,\,\,\,\,\,\,\,\,\, \thinspace t >0, \,\, x \in \Omega,  \\
& u(t,0)=u(t,1)=0,\\
& u(0)=\phi(0).
\end{split}
\right.
\end{equation}

Using \cite[Corollary 4.4]{CCJ2024}, since to $-K\leqslant \phi(0) \leqslant K$, we obtain the existence of  $u^+(t,K),$ and $u^-(t,-K)$, solutions \eqref{f+} and  \eqref{f-}, defined in $[0,\alpha^{-1}(\rho)]$, such that $u^-(t,-K)  \leqslant  u(t,\phi) \leqslant u^+(t,K)$ for $t\in [0,\alpha^{-1}(\rho)]$. Now, taking $K_1= \sup_{t\in [0,\alpha^{-1}(\rho)], x\in [0,1]}|u(t,\phi)(x)|$, we may repeat this procedure to ensure that $u(t,\phi)$ is defined in $[\alpha^{-1}(\rho),\alpha^{-1}(2\rho)]$ and, by induction, for all $t\geqslant 0$. To that, in each step, we use comparison in the interval $[\alpha^{-1}(i\rho),\alpha^{-1}((i+1)\rho)]$ to ensure that $u(t,\phi)$ is defined up to $\alpha^{-1}((i+1)\rho)$, knowing that condition $\mathbf{(S)}$ ensures that solutions $u^\pm(t-\alpha^{-1}(i\rho),\pm K_i)$ are defined for all $t\geqslant \alpha^{-1}(i\rho)$. This proves the global existence.
\end{proof}

This reasoning takes care of the global existence but, differently from the results in \cite{CCJ2024} the comparison has to be used in each step and does not give us boundedness of $u(t,\phi)$ from the boundedness of one solution $u(t,\pm K)$. The fact that for each interval $[\alpha^{-1}(i\rho),\alpha^{-1}((i+1)\rho)]$ we must use a different $K_i$ could result on unboundedness of $u(t,\phi)$ even if $u(t,\pm K_i)$ is bounded,  for each $i\in \mathbb{N}$.

To ensure boundedness of solutions we will need to use condition $\mathbf{(D)}$ in a step by step procedure like before. This condition ensures that the semigroup $\{e^{-(A+\nu C_0I)t}:t\geqslant 0\}$ generated by $-A-\nu C_0I$ is exponentially decaying, that is,  $\|e^{-(A+\nu C_0I)t}\|_{\mathcal{L}(L^2)}\leqslant e^{-\omega t}$, $t\geqslant 0$.

Let us start by estimating $u^+(t,K)$ (the estimate for $u^-(t,- K)$ is analogous). We know it satisfies
$$
u^+(t,+ K)\leqslant e^{-(A+\nu C_0I)t}K + \int_0^t e^{-(A+\nu C_0I)(t-s)} C_1^* ds,
$$
where $C_1^*=\left(\frac{\gamma K}{M}+ C_1\right)$. Now,  for $t\in [0,\alpha^{-1}(\rho)]$,
\begin{equation*}
\begin{split}
\|e^{-(A+\nu C_0I)t}K + \int_0^t e^{-(A+\nu C_0I)(t-s)} C_1^* ds\|_{L^\infty} &\leqslant\|e^{-(A+\nu C_0I)t}K + \int_0^t e^{-(A+\nu C_0I)(t-s)} C_1^* ds\|_{H^1_0(0,1)}\\
&\leqslant  e^{-\omega t} K + \int_0^t e^{-\omega(t-s)}C_1^* ds\\
& \leqslant  e^{-\omega  \alpha^{-1}(\rho)} K + \frac{C_1^*}{\omega}. 
\end{split}
\end{equation*}
Note that, we are free to choose $K$ and that $e^{-\omega \rho/m}+\frac{\gamma}{\omega m}<1$. Hence, we choose $K$ large  in such a way that
$$
 e^{-\omega  \alpha^{-1}(\rho)} K + \frac{C_1^*}{\omega}\leqslant  e^{-\omega \frac{\rho}{M} } K + \frac{C_1^*}{\omega}  <K.
$$
This ensures that, the ball of radius $K$ in $L^\infty(0,1)$ is an absorbing set for the solutions of \eqref{f+}. It follows that there is a bounded set in $H^1_0(0,1)$ that strongly absorbs bonded sets (the absorption time depends only on the elapsed time and does not depend on the initial time).

\begin{corol}\label{limitacionsol} Assume that $l:H^1_0(0,1)\to \mathbb{R}$, $a$, $f$ and $h$ are continuous functions, and assume that $f$ is such that, for every $r>0$, there is a constant $\kappa= \kappa(r) > 0 $ such that $ \kappa I + f(\cdot) $ is increasing in $[-r,r]$.
    \begin{itemize}
\item[(i)] If (\textbf{S}) holds, then the solutions of  \eqref{problema} are globally defined. 
\item[(ii)]If (\textbf{D}) holds, there are constants $K_\infty$ and $K^\alpha$  such that, for any $R>0$ there exists $T_R>0$ such that, for any $\phi\in C([-\rho,0],H^1_0(0,1))$ with $\|\phi\|_{C([-\rho,0],H^1_0(0,1))}\leqslant R$,
	\begin{equation}\label{k2_infty}
	\sup_{t\geqslant T_R }\|u^+(t,\phi)\|_{L^{\infty}(\Omega)} \leqslant K_\infty,
	\end{equation}	
	\begin{equation}\label{k2_r}
		\sup_{t\geqslant T_R }\|u^+(t,\phi)\|_{H^{1}_{0} } \leqslant K^\alpha.
	\end{equation}
Therefore, \eqref{k2_infty} and \eqref{k2_r} also hold for the solutions of  \eqref{problema}.
\item[(iii)] If  (\textbf{D}) holds, with the choice of $K$ as in the reasoning before the corollary, we define $C_1^*=\left(\frac{\gamma K}{M}+ C_1\right)$.  If $\phi \in C([-\rho,0],H^{1}_{0}(\Omega))$ and $\theta$ is the solution of 
\begin{equation}\label{A+C_0}
\left\{ \begin{array}{lcc}
(A+\nu C_0)\theta=C_1^* , \,\,\,\,\,\,\,\,\,\, \thinspace in\;\; \Omega,  \\
\\ \theta=0, \;\;\;\;\ in\;\; \partial \Omega,\\
\end{array}
\right.
\end{equation}
then $0 \leqslant \theta \in L^{\infty}(\Omega), \; \limsup_{t\rightarrow \infty} |u(t,\phi)(x)| \leqslant \theta(x)$ uniformly in $x\in \bar{\Omega}$ and for $\phi$ in bounded subsets of $C([-\rho,0],H^{1}_{0}(\Omega))$. Also, if $|\phi(t)(x)|\leqslant \theta(x)$ for all $x\in \bar{\Omega}$, $t\in [-\rho,0]$,  then $|u(t,\phi)(x)|\leqslant \theta(x)$ for all $x\in \bar{\Omega}$ and $t \geqslant 0$.
\end{itemize}
\end{corol}
\begin{proof}
    Follows from the reasoning preceding the corollary and $(iii)$ follows as in \cite{CCJ2024} for the appropriate choice of $C_1^*$.
\end{proof}

\section{Global Attractors of Multivalued Semiflows}\label{Satt}

In this section we consider $h= 0$ and introduce multivalued semiflow in order to study the asymptotic behavior of  solutions to \eqref{problema}.	
Before we can state the results, let us introduce the terminology of multivalued semiflows and their global attractors.

\bigskip

Let $(X,\|\cdot\|_X)$ be a Banach space and $P(X)$ ($\mathscr{B}(X)$)  be the collection of non-empty (bounded) subsets of $X$.
\begin{defin}\cite{melnik1998attractors}
A map $G:\mathbb{R}^+ \times X  \rightarrow P(X )$ is called a multivalued semiflow if the following conditions are satisfied:
\begin{itemize}
\item[(1)]$G(0,\cdot)=I$ is the identity map;
\item[(2)] $G(t_1+t_2,x) \subset G(t_1,G(t_2,x)), ~\forall t_1, t_2 \in \mathbb{R}^+, ~\forall x \in X ,$
\end{itemize} 
where, if $B\subset X$, $G(t,B)=\bigcup_{x \in B} G(t,x),\; B\subset X $.
\end{defin}
If  $(2)$ holds with equality we say that the multivalued semiflow is strict.

\bigskip

Let $C = C([0,\infty), X)$ and $\mathcal{R}\subset C$. Consider the conditions:

\begin{itemize}
\item[(C1)] For any $z \in X$ there exists at least one $\eta \in \mathcal{R} $ such that $\eta(0)=z$.
\item[(C2)] $\eta_{\tau}(\cdot):= \eta(\cdot + \tau) \in \mathcal{R}$ for any $\tau \geqslant 0$ and $\eta \in \mathcal{R}$ (translation propriety). 
\item[(C3)] Let $\eta,\psi \in \mathcal{R}$ be such that $\eta(0)= \psi(s)$ for some $s>0$. Then, the function $\theta$ defined by   
$$
\theta (t): = \left\{ \begin{array}{lc} \eta(t), &0 \leqslant t \leqslant s, \\  \psi (t-s), &  s \leqslant t, 
\end{array} \right. 
$$ 
belongs to $\mathcal{R}$ (concatenation property). 
\item[(C4)] For any sequence $\{\eta_j\}_{j \in \N} \subset \mathcal{R}$ such that $\eta_j(0) \rightarrow z$ in $X$, there exists a subsequence $\{\eta_{\mu}\}_{\mu \in \N}$ of $\{\eta_j\}_{j \in \N}$ and $\eta \in \mathcal{R}$ such that $\eta_{\mu}(t) \rightarrow \eta(t)$ for each $t \geqslant 0.$     
\end{itemize}
 
 Observe that $\mathcal{R}$ fulfilling (C1) and (C2) gives rise to a multivalued semiflow $G$ through \eqref{semiflow}  and if besides (C3) holds,
then this multivalued semiflow is strict (see \cite{caraballo2003comparison}, Proposition 2 or \cite{kapustyan2012global}, Lemma 9).

For $x \in X$, $A, B \subset X$, we set $\operatorname{dist}(x,B)=\inf\{\rho(x,y):y\in B\}.$ The Hausdorff semi-distance from $A$ to $B$ will be defined by $\operatorname{dist}_H(A,B)=\sup_{x \in A}\{\operatorname{dist}(x,B)\}.$
 
\begin{defin}\cite{melnik1998attractors}
We say that $A\subset X$ attracts $B\subset X$  under the action of the multivalued semiflow $G$,  if $\operatorname{dist}_H(G(t,B),A) \rightarrow 0$ as $t \rightarrow +\infty$. A set $M$ is said to be an attracting set for $G$ if $M$ attracts each set $B \in \mathscr{B}(X)$ under the action of the multivalued semiflow $G$. If there exists a bounded attracting set $M$ for $G$ we say that $G$ is bounded dissipative.
\end{defin}
 
For $A \subset X$, define $ \gamma_t^+(A)= \bigcup_{\tau \geqslant t}G(\tau,A)$, $t\geqslant 0$. The set $\gamma^+(A)=\gamma_0^+(A)$ is called the positive orbit of $A$ and the set
$\omega(A)=\bigcap_{t \geqslant 0}\overline{\gamma_t^+(A)}$ is called the omega limit ($\omega$-limit) set of $A$.

\begin{defin}\cite{melnik1998attractors} 
The set $A$ is said to be negatively semi-invariant if $A \subset G(t,A), \; \forall t \in \mathbb{R}^+.$
\end{defin}

\begin{defin}\cite{melnik1998attractors}
The set $\mathcal{A}$ is called a global attractor of the multivalued semiflow $G$ if it satisfies the next conditions: 
\begin{itemize}
\item[(1)] $\mathcal{A}$ attracts any $B \in \mathscr{B}(X)$;
\item[(2)] $\mathcal{A}$ is negatively semi-invariant, i.e. $\mathcal{A} \subset G(t, \mathcal{A}), \; \forall t \in \mathbb{R}^+.$
\end{itemize}
\end{defin}

\begin{defin}
The map $G(t,\cdot): X \rightarrow P(X)$ is upper semicontinuous if for every sequence $x_n \rightarrow x \in X$ we have $\lim_{n\to \infty}{\rm dist}_H(G(t,x_n),G(t,x))=0$.
\end{defin}

Let $X=C([-\rho,0],H^1_0(0,1))$ and if $w:[-\rho,\infty)$ is a solution of  \eqref{problema}, for each $\tau\in [0,\infty)$ denote $w^\tau(\theta)=w(\tau+\theta)$, $\theta\in [-\rho,0]$. The function $\eta(\cdot,\phi):[0,\infty)\to  X$ defined by $\eta(\tau,\phi)=w^\tau$, $\tau\geqslant 0$ is a continuous function and it is referred as a solution of the retarded differential problem \eqref{problema} with initial function $\phi\in X$. Denote $\mathcal{D}(\phi)$ the set of all solutions of the retarded differential problem \eqref{problema} with $h=0$ with initial function $\phi$. For each $\tau \geqslant 0$ we denote the map $G(\tau ,\cdot):X \rightarrow P(X)$ defined by:
\begin{equation}\label{semiflow}
    G(\tau,\phi)=\{\eta(\tau,\phi)\in X: \eta(\cdot,\phi) \hbox{ is in } \mathcal{D}(\phi)\}.
\end{equation}

\begin{theorem} \label{multi-semi}
Let  $f$\!,\! $a$,\! $l$,\! and $\phi$ be continuous functions and assume that $f$ is such that, for every $r>0$, there is a constant $\kappa= \kappa(r) > 0 $ such that $ \kappa I + f(\cdot) $ is increasing in $[-r,r]$ and  that condition {$\mathbf{(D)}$} is satisfied. Defining $\mathcal{R}$ as the set of all solutions of \eqref{problema}, then $\mathcal{R}$ defines an upper semicontinuous strict multivalued semiflow  in $X=C([-\rho,0],H^1_0(0,1))$.
\end{theorem}

\begin{proof}
Let $G:\mathbb{R}^+\times X\to P(X)$ be the map defined by \eqref{semiflow}. Property $(C1)$ follows from Theorem \ref{mild_problemau}, $(C2)$ and $(C3)$ are immediate consequences of our definition of solution (variation of constants formula). To prove property $(C4)$ consider a sequence
$\eta_j$ with $\eta_j\in \mathcal{D}(\phi_j)$ and $\phi_j\stackrel{j\to\infty}{\longrightarrow} \phi$ in $X$.

Let $w_j:[-r,\infty)\to H^1_0(0,1)$  be given by $w_j(\tau)=\phi_j(\tau)$ if $\tau\in [-\rho,0]$ and $w_j(\tau)=\eta_j(\tau,\phi_j)(0)$ for $\tau\geqslant 0$. 
 If $t=\int_{0}^{\tau}  a(l((w_j(r)))^{-1}dr=:\alpha^{-1}_j(\tau) $, let $\xi_j:[0,\infty)\to H^1_0(0,1)$ be the function given by $\xi_j(t)=w_j(\alpha(t))$.

For $0<\beta+\frac{1}{2}<1$ and $t \in ( (n-1)\rho,n\rho]$, $n\geqslant 1$ we have

\begin{equation*}
\begin{split}
&\| A^{\beta}\xi_j(t) \|_{H^{1}_{0}} \! \leqslant\! \| A^{\beta}S(t)\xi_j(0) \|_{H^{1}_{0}} \!+\! \displaystyle \int_{\alpha^{-1}((n-1)\rho)}^{t}\!\!\|A^{\beta}S(t\!-\!s)\dfrac{(\lambda f(\xi_j(s))\!+\!\gamma w_j(\alpha(s) \!-\!\rho))}{a(l(\xi_j(s)))}\|_{H^{1}_{0}} ds  \\
&+ \displaystyle \sum\limits_{m=1}^{n-1}\displaystyle \int_{ \alpha^{-1}((m-1)\rho)}^{\alpha^{-1}( m\rho)}\|A^{\beta}S(t-s)\dfrac{(\lambda f(\xi_j(s))+\gamma w_j(\alpha(s) \!-\!\rho))}{a(l(\xi_j(s)))}\|_{H^{1}_{0}}ds \\
&  \leqslant \! t^{-\beta} e^{-at}\|\xi_j(0)\|_{H^{1}_{0}} \!+\! N \displaystyle\! \int_{ \alpha^{-1}((n-1)\rho)}^{t}\!\!\!\!\! (t\!-\!s)^{-\frac{1}{2}\! \!-\beta}e^{-a(t\!-\!s)}ds \!+\! \displaystyle N \sum\limits_{m=1}^{n-1}\!\displaystyle \int_{\alpha^{-1}((m-1)\rho)}^{\alpha^{-1}(m\rho)} \!\!(t\!-\!s)^{-\frac{1}{2} \!-\!\beta}e^{-a(t\!-\!s)}ds\\
&  \leqslant  t^{-\beta} e^{-at}\|\xi_j(0)\|_{H^{1}_{0}} + N \displaystyle \int_{0}^\infty s^{-\frac{1}{2} -\beta}e^{- \pi^2  s}ds.
\end{split}
\end{equation*}

Since $X^{\frac{1}{2} + \beta } \subset \subset H^{1}_{0}(\Omega)$, it follows that the family of functions 
$\{\xi_j(t): \; \xi_j(\cdot) \in  \mathcal{D}(w_j), \; j \in \N \}$ is precompact in $H^{1}_{0}(\Omega)$.
Now, we are going to check that the family of functions  $\{ \xi_j \in \mathcal{R}, ~j \in \mathbb{N} \}$
is equicontinuous. In fact, for $r>0$ $ t_1 > 0$, if $t_1\!+\!r,~ t_1 \in ((n-1) \rho,  n\rho]$  we have
\begin{equation*}
\begin{split}
& \|\xi_j(t_1+r) - \xi_j(t_1) \|_{H^{1}_{0}}   \leqslant \| (S(t_1+r)-S(t_1))\xi_j(0)\|_{H^{1}_{0}} \\ 
& + \| \int_{\alpha^{-1}( (n-1)\rho)}^{t_1+r} S(t_1+r-s)\dfrac{(\lambda f(\xi_j(s))+\gamma w_j(\alpha(s) \!-\!\rho))}{a(l(\xi_j(s)))}ds  \\ 
& -\int_{\alpha^{-1}((n-1)\rho)}^{t_1} S(t_1-s)\dfrac{(\lambda f(\xi_j(s))+\gamma w_j(\alpha(s) \!-\!\rho))}{a(l(\xi_j(s)))}ds\|_{H^1_0(0,1)}  \\ 
& + \| \sum_{m=1}^{n-1} \int_{\alpha^{-1}((m-1)\rho)}^{\alpha^{-1}(m \rho)} S(t_1+r-s)\dfrac{(\lambda f(\xi_j(s))+\gamma w_j(\alpha(s) \!-\!\rho))}{a(l(\xi_j(s)))} \\ 
&- \sum_{m=1}^{n-1} \int_{\alpha^{-1}((m-1)\rho)}^{ \alpha^{-1}(m\rho) } S(t_1-s)\dfrac{(\lambda f(\xi_j(s))+\gamma w_j(\alpha(s) \!-\!\rho))}{a(l(\xi_j(s)))}\|_{H^1_0} \\ 
&  \leqslant \| (S(t_1)(S(r)-I)\xi_j(0)\|_{H^{1}_{0}}+  N  \int_{ \alpha^{-1}((n-1) \rho)}^{t_1} \|S(t_1+r-s) -  S(t_1-s) \|_{L(L^2,H^1_0)} ds  \\ 
& +  \int_{t_1}^{t_1+r}  \|S(t_1+r-s) \dfrac{(\lambda f(\xi_j(s))+\gamma w_j(\alpha(s) \!-\!\rho))}{a(l(\xi_j(s)))}\|_{H^{1}_{0}}ds  \\ 
&+ N  \sum_{m=1}^{n-1} \int_{ \alpha^{-1}((m-1)\rho)}^{ \alpha^{-1}(m\rho)}\|S(t_1+r-s) -  S(t_1-s) \|_{L(L^2,H^1_0)} ds \\ 
&\leqslant  M\|(S(r)-I)\xi_j(0)\|_{H^{1}_{0}}  + N  \int_{ \alpha^{-1}((n-1)\rho)}^{t_1} \|(S(r)-I) A^{-\beta}\|_{L(L^2)} \|A^{\beta+ \frac{1}{2} }S( t_1-s )\|_{L(L^2)}ds 
\\ & + NM' \int_{t_1}^{t_1+r} (t_1+r- s) ^{-\frac{1}{2}} e^{-a(t_1+r-s)}ds
\\ & +N \sum_{m=1}^{n-1} \int_{ \alpha^{-1}((m-1)\rho)}^{ \alpha^{-1}(m\rho)}\|(S(t_1+r-s)-I) A^{-\beta}\|_{L(L^2)} \|A^{\beta+ \frac{1}{2} }S( t_1-s )\|_{L(L^2)}ds \\
&\leqslant  M\|(S(r)-I)\xi_j(0)\|_{H^{1}_{0}} 
 + NK  \int_{0 }^{t_1} r^{\beta} (t_1 -s)^{-\frac{1}{2}-\beta} e^{-a(t_1 - s)} ds \\ &
 + NK'  \int_{t_1}^{t_1+r} (t_1+r- s) ^{-\frac{1}{2}} e^{-a(t_1+r-s)}ds,
 \\ & \leqslant  M\|(S(r)-I)\xi_j(0)\|_{H^{1}_{0}} + r ^{\beta} NK  \int_{0}^{\infty}  s^{-\frac{1}{2}-\beta} e^{-a s} ds 
 + NK'  \int_{0}^{r} s^{-\frac{1}{2}} e^{-as}ds,
\end{split}
\end{equation*}
This last expression is independent of $\xi_j, \; j \in \N$ and tends to zero as $r \rightarrow 0$  since $S(t)$ is a $C_0$ semigroup.
From Corollary \ref{limitacionsol}
\[
\sup_{t \in [0,\infty), \; j \in \N} \|\xi_j(t)\|_{H^{1}_{0}}  < C,
\]
then (C4) follows from the Arzel\`a-Ascoli Theorem applied to the sequencee $\{\eta_j\}$.

To prove that the strict multivalued semiflow $G$  is upper semicontinuous we note that if $\phi_j \stackrel{j\to\infty}{\longrightarrow} \phi_0$ in $C([-\rho,0],H^1_0(0,1))$ and $\eta(\cdot,\phi_j)\in \mathcal{D}(\phi_j)$ are such that ${\rm dist}(\eta(\tau,\phi_j),G(\tau,\phi_0))\stackrel{j\to\infty}{\not\longrightarrow}0$,
there exists $\epsilon>0$ and an infinite set $\N_1\subset \N$ such that  ${\rm dist}(\eta(\tau,\phi_j),G(\tau,\phi_0))\geq \epsilon$ for all $j\in \N_1$. Since $\{\eta(\cdot,\phi_j): j\in \N_1\}\subset C([-\rho,\infty),H^1_0(0,1))$ has a convergent subsequence to a solution $\eta(\cdot,\phi_0)$, we have that $\eta(\tau,\phi_0)\in G(t,\phi_0)$, which is a contradiction.
\end{proof}

\begin{defin}\cite{melnik1998attractors}\label{def:existsattractor}
A multivalued semiflow $G$ is said to be asymptotically upper semicompact if for all $B\in  \mathscr{B}(X)$ such that, for some $T(B)>0$, $\gamma_{T(B)}^+(B)\in  \mathscr{B}(X)$, any sequence $\xi_n\in G(t_n,B)$ with $t_n\stackrel{n\to\infty}{\longrightarrow}\infty$ has a convergent subsequence.
\end{defin}

\begin{theorem}\cite{melnik1998attractors}\label{teo:existsattractor}
If $G$ is an asymptotically upper semicompact multivalued semiflow which is bounded dissipative and such that $G(\tau,\cdot): X \rightarrow P(X)$ is upper semicontinuous,  for each $\tau\geqslant 0$, then $G$ has the global attractor $\mathcal{A}$.
\end{theorem}

\begin{theorem} \label{Tattractor}
Assume that $f,~a\circ l$ are continuous, $h\equiv 0$ and $f$ is such that, for every $r>0$, there is a constant $\kappa= \kappa(r) > 0 $ such that $ \kappa I + f(\cdot) $ is increasing in $[-r,r]$  and that condition $\mathbf{(D)}$ is satisfied. Then, the multivalued semiflow $G$  defined in Theorem \ref{multi-semi} has a global attractor.
\end{theorem}

\begin{proof}
We only need to prove that $G$ is asymptotically upper semicompact and bounded dissipative. Bounded dissipativeness follows from Corollary \ref{limitacionsol}. As for the property that $G$ is asymptotically upper semicompact it follows from estimate \eqref{k2_r} in  Corollary \ref{limitacionsol} and from the compact embedding of $X^\alpha$ into $H^1_0(0,1)$ for $\alpha \in (\frac12,1)$.
\end{proof}

We introduce the following definition that will be necessary to enunciate a standard well-known  result for describing the attractor as the union of bounded complete trajectories reads in the multivalued case as follows.
\begin{defin}
A map $\gamma: \mathbb{R}\rightarrow H^{1}_{0}(\Omega) $ is called global solution of $\mathcal{R}$ (resp. of $G$) if 
$\gamma(\cdot + h)|_{[0,\infty)} \in \mathcal{R}$ for all $h \in \mathbb{R}$ (resp. if $\gamma(\tau + \sigma) \in G(\tau,\gamma(\sigma))$ for all $\sigma \in \mathbb{R}$ and $\tau \geqslant 0).$
\end{defin} 
\begin{theorem}\cite{caballero2019robustness}\label{standartglobalattractor}
Consider $\mathcal{R}$ satisfying (C1) and (C2), and either (C3) or (C4).\\
Assume also that $G$ possesses a compact global attractor $\mathcal{A}$. Then  
$$
\mathcal{A} = \{\gamma(0): ~\gamma \in \mathbb{K}\}= \bigcup_{\tau \in \mathbb{R}}\{\gamma(\tau): ~\gamma \in \mathbb{K}\}
$$
where $\mathbb{K}$ denotes the set of all bounded global solutions in $\mathcal{R}$.
\end{theorem}
 
\begin{theorem}\label{attractorglobal}
Assume that $f,~a,~l$ are continuous, $h\equiv 0$ and $f$ is such that, for every $r>0$, there is a constant $\kappa= \kappa(r) > 0 $ such that $ \kappa I + f(\cdot) $ is increasing in $[-r,r]$ and that condition $\mathbf{(D)}$ is satisfied. Then, the multivalued semiflow $G$ defined in Theorem \ref{multi-semi} has a global attractor $\mathcal{A}$ in $X$ such that 
\begin{equation*}
\begin{split}
&\mathcal{A} \subset \Sigma(\phi):=\{\phi \in C([-\rho,0], L^\infty(\Omega)),\; |\phi(\sigma)(x)| \leqslant \theta(x),\ x \in \Omega, \sigma \in [-\rho,0]\},
\end{split}
\end{equation*}
where $\theta$ is the function given in Corollary \ref{limitacionsol} (iii). If  \eqref{modulusofcont} holds then 
\begin{equation*}
\begin{split}
&\mathcal{A} \subset C([-\rho,0],H^2(0,1)\cap H^1_0(0,1)).
\end{split}
\end{equation*}
\end{theorem}

\begin{proof}

By Theorem \ref{Tattractor}, $G$ has a global attractor $\mathcal{A}$. Note that, by Corollary \ref{limitacionsol}, $\Sigma(\phi)$ is attracting for $G$. Then,  $\mathcal{A} \subset \Sigma(\phi)\cap C([-\rho,0],X^\alpha)$, $\alpha<1$. The regularity of the solutions in the global attractor $\mathcal{A}$ follows from Theorem \ref{T_regularity}.
\end{proof}

\bibliographystyle{acm}

\end{document}